\documentclass[11pt]{article}
\usepackage{amssymb,amsmath,amsthm}
\usepackage{enumerate}
\usepackage{pgf,tikz}
\usetikzlibrary{arrows}

\newtheorem{theorem}{Theorem}[section]
\newtheorem{corollary}[theorem]{Corollary}
\newtheorem{lemma}[theorem]{Lemma}
\newtheorem{proposition}[theorem]{Proposition}

\newtheorem{example}[theorem]{Example}
\newtheorem{remark}[theorem]{Remark}
\numberwithin{equation}{section}

\DeclareMathOperator{\co}{co}
\DeclareMathOperator{\Des}{Des}

\newcommand{\SSSS}{\mathfrak S}

\newcommand{\cupdots}{\cup \cdots \cup}

\DeclareMathOperator{\carries}{carries}
\DeclareMathOperator{\Orb}{Orb}
\DeclareMathOperator{\Comp}{Comp}
\DeclareMathOperator{\NCP}{NCP}

\tikzstyle{internal node}=[circle,draw,inner sep=1.5,fill=black]

\newcommand{\TablePisTwo}
{
\begin{table}

\[\def\arraystretch{1.6}
\begin{array}{|r|r|r|r|r|r|r|r|r|r|r|r|}
  \hline
    \mathbf{n}		& \mathbf{n}_2  & d_2(\mathbf{n}) & r=1	& r=2 	& r=3 	& r=4	& r=5	& r=6    & r=7  & r=8    & r=9   	\\
  \hline\hline

	\mathbf{2}		& 10		& 0		& {\color{red}1},1		& {\color{red}1},1		& {\color{red}1},1		& {\color{red}1},1		& {\color{red}1},1		& {\color{red}1},1		& {\color{red}1},1		& {\color{red}1},1		& {\color{red}1},1		\\
	\mathbf{3}		& 11		& 1		& {\color{blue}1},1		& {\color{blue}1},1		& {\color{orange}1},1		& {\color{blue}1},1		& {\color{orange}1},1		& {\color{orange}1},1		& {\color{orange}1},1		& {\color{blue}1},1		& {\color{orange}1},1		\\
	\mathbf{4}		& 100		& 0		& {\color{red}3},3		& {\color{red}3},3		& {\color{red}3},3		& {\color{red}3},3		& {\color{red}3},3		& {\color{red}3},3		& {\color{red}3},3		& {\color{red}3},3		& {\color{red}3},3		\\
	\mathbf{5}		& 101		& 1		& {\color{blue}3},3		& {\color{blue}3},6		& {\color{blue}2},2		& {\color{blue}3},3		& {\color{blue}2},2		& {\color{green!60!black}3},3		& {\color{orange}2},2		& {\color{blue}3},3		& {\color{blue}2},2		\\
	\mathbf{6}		& 110		& 1		& {\color{blue}4},4		& {\color{blue}4},4		& {\color{blue}3},3		& {\color{blue}4},4		& {\color{blue}3},3		& {\color{blue}3},5		& {\color{orange}3},3		& {\color{blue}4},4		& {\color{blue}3},3		\\
	\mathbf{7}		& 111		& 2		& {\color{blue}4},4		& {\color{blue}3},3		& {\color{blue}1},2		& {\color{orange}2},5		& {\color{orange}2},2		& {\color{green!60!black}3},4		& {\color{orange}2},2		& {\color{green!60!black}4},5		& {\color{orange}2},2		\\
	\mathbf{8}		& 1000		& 0		& {\color{red}7},7		& {\color{red}7},7		& {\color{red}7},7		& {\color{red}7},7		& {\color{red}7},7		& {\color{red}7},7		& {\color{red}7},7		& {\color{red}7},7		& {\color{red}7},7		\\
	\mathbf{9}		& 1001		& 1		& {\color{blue}7},7		& {\color{blue}7},7		& {\color{blue}6},7		& {\color{blue}7},8		& {\color{blue}6},6		& {\color{blue}6},6		& {\color{blue}5},7		& {\color{blue}7},7		& {\color{blue}6},6		\\
	\mathbf{10}		& 1010		& 1		& {\color{blue}8},8		& {\color{blue}8},8		& {\color{blue}7},7		& {\color{blue}8},8		& {\color{blue}7},7		& {\color{blue}7},8		& {\color{blue}6},10		& {\color{blue}8},8		& {\color{blue}7},7		\\
	\mathbf{11}		& 1011		& 2		& {\color{blue}8},8		& {\color{blue}7},8		& {\color{blue}5},5		& {\color{blue}5},6		& {\color{blue}3},6		& {\color{green!60!black}3},5		& {\color{orange}3},5		& {\color{green!60!black}4},7		& {\color{green!60!black}4},9		\\
	\mathbf{12}		& 1100		& 1		& {\color{blue}10},10		& {\color{blue}10},10		& {\color{blue}9},9		& {\color{blue}10},10		& {\color{blue}9},9		& {\color{blue}9},11		& {\color{blue}8},12		& {\color{blue}10},10		& {\color{blue}9},9		\\
	\mathbf{13}		& 1101		& 2		& {\color{blue}10},10		& {\color{blue}9},9		& {\color{blue}7},7		& {\color{blue}7},10		& {\color{blue}5},7		& {\color{blue}4},7		& {\color{orange}3},6		& {\color{green!60!black}4},11		& {\color{green!60!black}4},10		\\
	\mathbf{14}		& 1110		& 2		& {\color{blue}11},11		& {\color{blue}10},11		& {\color{blue}8},10		& {\color{blue}8},13		& {\color{blue}6},8		& {\color{blue}5},9		& {\color{blue}3},7		& {\color{blue}4},11		& {\color{green!60!black}4},9		\\
	\mathbf{15}		& 1111		& 3		& {\color{blue}11},11		& {\color{blue}9},9		& {\color{blue}6},7		& {\color{blue}5},8		& {\color{green!60!black}3},7		& {\color{green!60!black}3},8		& {\color{orange}3},6		& {\color{green!60!black}4},8		& {\color{green!60!black}4},8		\\
	\mathbf{16}		& 10000		& 0		& {\color{red}15},15		& {\color{red}15},15		& {\color{red}15},15		& {\color{red}15},15		& {\color{red}15},15		& {\color{red}15},15		& {\color{red}15},15		& {\color{red}15},15		& {\color{red}15},15		\\
	\mathbf{17}		& 10001		& 1		& {\color{blue}15},15		& {\color{blue}15},15		& {\color{blue}14},15		& {\color{blue}15},17		& {\color{blue}14},14		& {\color{blue}14},14		& {\color{blue}13},14		& {\color{blue}15},15		& {\color{blue}14},14		\\
	\mathbf{18}		& 10010		& 1		& {\color{blue}16},16		& {\color{blue}16},16		& {\color{blue}15},15		& {\color{blue}16},16		& {\color{blue}15},15		& {\color{blue}15},17		& {\color{blue}14},15		& {\color{blue}16},16		& {\color{blue}15},15		\\
	\mathbf{19}		& 10011		& 2		& {\color{blue}16},16		& {\color{blue}15},15		& {\color{blue}13},14		& {\color{blue}13},13		& {\color{blue}11},12		& {\color{blue}10},16		& {\color{blue}8},11		& {\color{blue}9},14		& {\color{blue}7},13		\\
	\mathbf{20}		& 10100		& 1		& {\color{blue}18},18		& {\color{blue}18},18		& {\color{blue}17},17		& {\color{blue}18},18		& {\color{blue}17},17		& {\color{blue}17},19		& {\color{blue}16},18		& {\color{blue}18},18		& {\color{blue}17},17		\\
\hline
\end{array}
\]
\caption{A comparison of our best prediction of the number 
of factors of $2$ in $A_{n}^{r}$ with the actual number.
Predictions are given first, colored according to whether
the result is given by
{\color{orange} Proposition~\ref{proposition_Euler_p}},
{\color{green!60!black} Proposition~\ref{proposition_Euler_2}}, 
{\color{red} Theorem~\ref{theorem_even_prime}},  or
{\color{blue} Theorem~\ref{theorem_better_bound_2}},
and the actual value is given second.}
\label{table_p_2}
\end{table}
}

\newcommand{\TablePisThree}
{
\begin{table}

\[\def\arraystretch{1.6}
\begin{array}{|r|r|r|r|r|r|r|r|r|r|r|r|}
  \hline
    \mathbf{n}		& \mathbf{n}_3  & d_3(\mathbf{n}) & r=1	& r=2 	& r=3 	& r=4	& r=5	& r=6    & r=7  & r=8  & r=9   	\\
  \hline\hline

	\mathbf{2}		& 2		& 1		& {\color{blue}0},0		& {\color{blue}0},0		& {\color{blue}0},0		& {\color{blue}0},0		& {\color{green!60!black}0},0		& {\color{blue}0},0		& {\color{green!60!black}0},0		& {\color{blue}0},0		& {\color{blue}0},0		\\
	\mathbf{3}		& 10		& 0		& {\color{green!60!black}1},1		& {\color{blue}0},0		& {\color{green!60!black}2},2		& {\color{blue}0},0		& {\color{red}1},1		& {\color{blue}0},0		& {\color{red}1},1		& {\color{blue}0},0		& {\color{green!60!black}3},3		\\
	\mathbf{4}		& 11		& 1		& {\color{blue}1},1		& {\color{blue}0},0		& {\color{blue}1},2		& {\color{blue}0},0		& {\color{orange}1},1		& {\color{blue}0},0		& {\color{orange}1},1		& {\color{blue}0},0		& {\color{orange}2},3		\\
	\mathbf{5}		& 12		& 2		& {\color{blue}1},1		& {\color{blue}0},0		& {\color{orange}1},1		& {\color{blue}0},0		& {\color{orange}1},1		& {\color{blue}0},0		& {\color{orange}1},2		& {\color{blue}0},0		& {\color{orange}1},1		\\
	\mathbf{6}		& 20		& 1		& {\color{blue}2},2		& {\color{blue}0},0		& {\color{blue}2},4		& {\color{blue}0},0		& {\color{blue}1},2		& {\color{blue}0},0		& {\color{brown}2},2		& {\color{blue}0},0		& {\color{blue}2},5		\\
	\mathbf{7}		& 21		& 2		& {\color{blue}2},2		& {\color{blue}0},0		& {\color{blue}1},2		& {\color{blue}0},0		& {\color{orange}1},1		& {\color{blue}0},0		& {\color{orange}1},1		& {\color{blue}0},0		& {\color{orange}2},3		\\
	\mathbf{8}		& 22		& 3		& {\color{blue}2},2		& {\color{blue}0},0		& {\color{orange}1},2		& {\color{blue}0},0		& {\color{orange}1},2		& {\color{blue}0},0		& {\color{brown}2},2		& {\color{blue}0},0		& {\color{orange}2},2		\\
	\mathbf{9}		& 100		& 0		& {\color{green!60!black}4},4		& {\color{blue}0},0		& {\color{green!60!black}5},6		& {\color{blue}0},0		& {\color{red}4},4		& {\color{blue}0},0		& {\color{red}4},4		& {\color{blue}0},0		& {\color{green!60!black}6},7		\\
	\mathbf{10}		& 101		& 1		& {\color{blue}4},4		& {\color{blue}0},0		& {\color{blue}4},6		& {\color{blue}0},0		& {\color{blue}3},4		& {\color{blue}0},0		& {\color{blue}3},4		& {\color{blue}0},0		& {\color{blue}4},6		\\
	\mathbf{11}		& 102		& 2		& {\color{blue}4},4		& {\color{blue}0},0		& {\color{blue}3},3		& {\color{blue}0},0		& {\color{blue}1},2		& {\color{blue}0},0		& {\color{brown}2},2		& {\color{blue}0},0		& {\color{orange}2},4		\\
	\mathbf{12}		& 110		& 1		& {\color{blue}5},5		& {\color{blue}0},0		& {\color{blue}5},6		& {\color{blue}0},0		& {\color{blue}4},5		& {\color{blue}0},0		& {\color{blue}4},5		& {\color{blue}0},0		& {\color{blue}5},7		\\
	\mathbf{13}		& 111		& 2		& {\color{blue}5},5		& {\color{blue}0},0		& {\color{blue}4},4		& {\color{blue}0},0		& {\color{blue}2},3		& {\color{blue}0},0		& {\color{brown}2},3		& {\color{blue}0},0		& {\color{brown}3},5		\\
	\mathbf{14}		& 112		& 3		& {\color{blue}5},5		& {\color{blue}0},1		& {\color{blue}3},4		& {\color{purple}1},2		& {\color{orange}1},2		& {\color{purple}1},1		& {\color{brown}2},2		& {\color{purple}1},1		& {\color{orange}2},5		\\
	\mathbf{15}		& 120		& 2		& {\color{blue}6},6		& {\color{blue}0},0		& {\color{blue}5},5		& {\color{blue}0},0		& {\color{blue}3},5		& {\color{blue}0},0		& {\color{blue}2},5		& {\color{blue}0},0		& {\color{brown}3},6		\\
	\mathbf{16}		& 121		& 3		& {\color{blue}6},6		& {\color{purple}1},1		& {\color{blue}4},5		& {\color{purple}1},1		& {\color{blue}1},3		& {\color{purple}1},1		& {\color{brown}2},3		& {\color{purple}1},1		& {\color{brown}3},7		\\
	\mathbf{17}		& 122		& 4		& {\color{blue}6},6		& {\color{blue}0},0		& {\color{blue}3},4		& {\color{blue}0},0		& {\color{orange}1},2		& {\color{blue}0},0		& {\color{brown}2},2		& {\color{blue}0},0		& {\color{orange}2},3		\\
	\mathbf{18}		& 200		& 1		& {\color{blue}8},8		& {\color{blue}0},0		& {\color{blue}8},11		& {\color{blue}0},0		& {\color{blue}7},8		& {\color{blue}0},0		& {\color{blue}7},8		& {\color{blue}0},0		& {\color{blue}8},12		\\
	\mathbf{19}		& 201		& 2		& {\color{blue}8},8		& {\color{blue}0},0		& {\color{blue}7},8		& {\color{blue}0},0		& {\color{blue}5},7		& {\color{blue}0},0		& {\color{blue}4},7		& {\color{blue}0},0		& {\color{blue}4},10		\\
	\mathbf{20}		& 202		& 3		& {\color{blue}8},8		& {\color{blue}0},0		& {\color{blue}6},8		& {\color{blue}0},0		& {\color{blue}3},5		& {\color{blue}0},0		& {\color{brown}2},6		& {\color{blue}0},0		& {\color{brown}3},7		\\
\hline
\end{array}
\]
\caption{A comparison of our best prediction of the number 
of factors of $3$ in $A_{n}^{r}$ with the actual number.
Predictions are given first, colored according to whether
the result is given by
{\color{orange} Proposition~\ref{proposition_Euler_p}},
{\color{brown} Example~\ref{example_recursion_3}},
{\color{purple} Example~\ref{example_3_14}},
{\color{red} Theorem~\ref{theorem_odd_prime}},
{\color{green!60!black} Proposition~\ref{proposition_simple_shift}},  or
{\color{blue} Theorem~\ref{theorem_group_odd_p}},
and the actual value is given second.}
\label{table_p_3}
\end{table}
}

\begin{document}

\author{Richard Ehrenborg and Alex Happ}
\title{On the powers of the descent set statistic}
\date{}

\maketitle

\begin{abstract}
We study the sum of the $r$th powers of the descent set statistic
and how many small prime factors occur in these numbers.
Our results depend upon the base $p$ expansion of $n$ and $r$.
\end{abstract}

\section{Introduction}

It has always been interesting to study
divisibility properties of sequences defined 
combinatorially.
Three classical examples are
Fibonacci numbers,
the partition function,
and
binomial coefficients.
The Fibonacci numbers satisfy
$\gcd(F_{m},F_{n}) = F_{\gcd(m,n)}$.
Ramanujan discovered that
the partition function satisfies, among other relations,
that $5$ divides $p(5n+4)$.
The binomial coefficients 
are well-studied modulo a prime;
see the theorems of Lucas and Kummer
in Section~\ref{section_preliminaries}.
In this paper we consider divisibility properties
of the sum of powers of the descent set statistic
from permutation enumeration.
The descent set statistic was first studied by
MacMahon~\cite{MacMahon}.

For a permutation $\pi$ in the symmetric group~$\SSSS_{n}$,
the descent set
of $\pi$ is the subset of $[n-1] = \{1,2, \ldots, n-1\}$
given by
$\Des(\pi) = \{i \in [n-1] \: : \: \pi_{i} > \pi_{i+1}\}$.
The descent set statistics~$\beta_{n}(S)$
are defined for subsets $S$ of $[n-1]$ by
$$  \beta_{n}(S)
   =
      \left|
        \{ \pi \in \SSSS_{n} \: : \: \Des(\pi) = S \}
      \right|   .   $$
Since there are $n!$ permutations, we directly have
$$  n! = \sum_{S \subseteq [n-1]} \beta_{n}(S)  . $$      
Define $A^{r}_{n}$ to be the sum of the $r$th powers of the descent
set statistics, that is,
$$  A^{r}_{n} = \sum_{S \subseteq [n-1]} \beta_{n}(S)^{r}  . $$
This quantity occurs naturally
as moments of the random variable
$\Des(S)$, where the set $S$ is chosen with a uniform
distribution from all subsets of the set $[n-1]$.

In Section~\ref{section_divisibility}
we give two expressions,
depending on the parity
of $r$ for $A^{r}_{n}$;
see Lemma~\ref{lemma_expansion}.
We continue by showing that for an odd prime $p$ and
an even positive integer $r$,
if $m$ and $n$ contain the same non-zero digits in base $p$,
then
the prime $p$ dividing $A_{m}^{r}$ is equivalent to
$p$ dividing $A_{n}^{r}$.
In Section~\ref{section_number_of_prime_factors}
we give lower bounds for the number of prime factors
in $A^{r}_{n}$. These bounds depend on
the digit sum of $n$ in base $p$.
Unfortunately, we do not obtain any bound when $p$ is
an odd prime and $r$ is even.
In Section~\ref{section_improving_the_bound}
we sharpen the results by
collecting terms together
occurring in the expansion
of Lemma~\ref{lemma_expansion}.
The method of collection is by considering
orbits of a group action.
First we use the cyclic group ${\mathbb Z}_{p^{k}}$,
and then we use a group defined by 
the action on the balanced $p$-ary tree of cyclically
rotating the branches under any node.
The lower bounds obtained in this section for the prime factors of $p$
in $A^{r}_{n}$ now also depend on the base $p$ expansion of $r$.

We end in the concluding remarks by presenting
two tables obtained by computation to compare our
bounds with the actual number of factors of $2$ and $3$
occurring in $A^{r}_{n}$.

\section{Preliminaries}
\label{section_preliminaries}
      
Define $\alpha_{n}(S)$ by the sum
$$
\alpha_{n}(S) = \sum_{T \subseteq S} \beta_{n}(T) .
$$      
Observe that $\alpha_{n}(S)$ enumerates the number of permutations
in $\SSSS_{n}$ with descent set contained in the set $S$.
Especially, we know that
$A^{1}_{n} = \alpha_{n}([n-1]) = n!$.
For more on descents; see~\cite[Section~1.4]{EC1}.

Define a bijection $\co$
from subsets of the set $[n-1]$
to compositions of $n$ by sending the set
$S = \{s_{1} < s_{2} < \cdots < s_{k-1}\}$
to the composition
$\co(S) = (c_{1}, c_{2}, \ldots, c_{k})$, where 
$c_{i} = s_{i} - s_{i-1}$ with $s_{0} = 0$ and $s_{k} = n$.
See, for instance,~\cite{C_E_P_R}
or~\cite[Section~7.19]{EC2}.
It is now straightfoward to observe that
$\alpha_{n}(S)$ is given by the multinomial coefficient
$\binom{n}{\co(S)}$.

Using elementary number theory we have three observations.
 \begin{proposition}
Let $p$ be a prime.
Assume that $r$ and $s$ are both greater than or equal to $k$ and
$r \equiv s \bmod p^{k-1} \cdot (p-1)$.
Then the congruence
$A^{r}_{n} \equiv A^{s}_{n} \bmod p^{k}$ holds.
Especially, the statement $p^{k}$ divides~$A^{r}_{n}$ 
is equivalent to $p^{k}$ divides $A^{s}_{n}$.
\label{proposition_Euler_p}
\end{proposition}
\begin{proof}
We may assume that $r < s$,
that is, $s-r = p^{k-1} \cdot (p-1) \cdot j$ for a positive integer $j$.
For an integer $x$ which is relative prime to the prime $p$,
Euler's theorem implies that
$x^{s} \equiv x^{r} \cdot (x^{p^{k-1} (p-1)})^{j} \equiv x^{r} \bmod p^{k}$.
For an integer $x$ which is divisible by the prime $p$,
we have
$x^{s} \equiv 0 \equiv x^{r} \bmod p^{k}$
since $r,s \geq k$.
Thus
for all integers $x$ we have
$x^{s} \equiv x^{r} \bmod p^{k}$
and 
we conclude
$A^{s}_{n}
\equiv 
\sum_{S \subseteq [n-1]} \beta_{n}(S)^{s}
\equiv 
\sum_{S \subseteq [n-1]} \beta_{n}(S)^{r}
\equiv 
A^{r}_{n} \bmod p^{k}$.
\end{proof}

When the prime $p$ is $2$ and $k \geq 3$, we have an improvement
of a factor of $2$.
\begin{proposition}
Assume that $r$ and $s$ are both greater than or equal to $k \geq 3$ and
$r \equiv s \bmod 2^{k-2}$.
Then the congruence
$A^{r}_{n} \equiv A^{s}_{n} \bmod 2^{k}$ holds.
Especially, the statement $2^{k}$ divides~$A^{r}_{n}$ 
is equivalent to $2^{k}$ divides $A^{s}_{n}$.
\label{proposition_Euler_2}
\end{proposition}
\begin{proof}
For an odd integer $x$ we know that
$x^{2^{k-2}} \equiv 1 \bmod 2^{k}$,
which yields the better bound using the same
argument as in the proof of Proposition~\ref{proposition_Euler_p}.
\end{proof}

\begin{proposition}
Let $p$ be a prime and $r$ an integer such that $r \geq k \cdot p$.
If $p^{k}$ divides the $k$ numbers
$A^{r - (p-1)}_{n}$,
$A^{r- 2 \cdot (p-1)}_{n}$,
through
$A^{r- k \cdot (p-1)}_{n}$,
then
$p^{k}$ divides $A^{r}_{n}$.
\label{proposition_recursion}
\end{proposition}
\begin{proof}
By Fermat's little theorem we know
$x^{p-1} - 1 \equiv 0 \bmod p$
for $x$ relative prime to $p$.
Hence the $k$th power of this quantity
is divisible by $p^{k}$, that is,
$(x^{p-1} - 1)^{k} \equiv 0 \bmod p^{k}$.
Note that
$x^{k} \equiv 0 \bmod p^{k}$
for $x$ not relative prime to $p$.
Multiplying these two statements we obtain
$$ 
x^{k \cdot p}
-
\binom{k}{1} \cdot x^{k \cdot p - (p-1)}
+ 
\cdots
+
(-1)^{k} \cdot x^{k}
\equiv
0 \bmod p^{k}
$$
for all $x$.
Multiply this polynomial relation with $x^{r - k \cdot p}$,
substitute $x$ to be $\beta_{n}(S)$, and sum over all
$S \subseteq [n-1]$ to obtain the linear recursion
$$ 
A^{r}_{n}
-
\binom{k}{1} \cdot A^{r - (p-1)}_{n}
+ 
\cdots
+
(-1)^{k} \cdot A^{r - k \cdot (p-1)}_{n}
\equiv
0 \bmod p^{k} .
$$
This relation yields the result.
\end{proof}

\begin{example}
{\rm
Note using Table~\ref{table_p_2} that for $8 \leq n \leq 20$, 
the power $2^5$ divides $A^{r}_{n}$ 
when $5\leq r\leq 9$. Hence,
Proposition~\ref{proposition_recursion} gives that 
$2^5$ divides $A^{r}_{n}$ for $r \geq 5$.
}
\label{example_recursion_2}
\end{example}

\begin{example}
{\rm
Using Table~\ref{table_p_3}
we know for $n = 6$ and $8 \leq n \leq 20$
that $3^{2}$ divides $A^{3}_{n}$ and~$A^{5}_{n}$.
Hence,
Proposition~\ref{proposition_recursion}
implies  
for $r$ odd and $r \geq 3$ that
$3^{2}$ divides $A^{r}_{n}$.
Similarly, we know 
for $n \in \{9,10,12,13,15,16,18,19,20\}$
that $3^{3}$ divides $A^{3}_{n}$, $A^{5}_{n}$ and $A^{7}_{n}$.
Therefore, 
for these same values of $n$,
$3^{3}$ divides $A^{r}_{n}$ for $r$ odd and $r \geq 3$.
}
\label{example_recursion_3}
\end{example}

\begin{remark}
{\rm
Note that Propositions~\ref{proposition_Euler_p} 
through~\ref{proposition_recursion} apply to any
sequence of the form
$\sum_{i=1}^{N} c_{i} \cdot d_{i}^{r}$
where $c_{i}$ and $d_{i}$ are integers.
}
\end{remark}

We end this section by reviewing 
Lucas' theorem, see~\cite[Chapter~XXIII, Section~228]{Lucas},
and Kummer's theorem,
see~\cite{Kummer},
for multinomial coefficients.
\begin{theorem}[Lucas]
Let $p$ be a prime
and $\vec{c} = (c_{1}, c_{2}, \ldots, c_{k})$ be a weak composition of $n$,
that is, $0$ is allowed as an entry.
Expand $n$ and each $c_{i}$
in base $p$, that is,
$n = \sum_{j \geq 0} n_{j} \cdot p^{j}$
and
$c_{i} = \sum_{j \geq 0} c_{i,j} \cdot p^{j}$
where
$0 \leq n_{j}, c_{i,j} \leq p-1$.
Let
$\vec{c}_{j}$ be the weak composition
$\vec{c}_{j} = (c_{1,j}, c_{2,j}, \ldots, c_{k,j})$.
Then the multinomial coefficient
$\binom{n}{\vec{c}}$ modulo $p$
is given by
$$
\binom{n}{\vec{c}}
\equiv
\prod_{j \geq 0}
\binom{n_{j}}{\vec{c}_{j}}
\bmod p  .
$$
\label{theorem_Lucas}
\end{theorem}

Let $\carries_{p}(\vec{c})$ denote the number of carries when
adding $c_{1} + c_{2} + \cdots + c_{k}$ in base $p$.

\begin{theorem}[Kummer]
For a prime $p$ and a composition
$\vec{c} = (c_{1}, c_{2}, \dots, c_{k})$ of $n$,
the largest power~$d$ such that $p^{d}$ 
divides the multinomial coefficient $\binom{n}{\vec{c}}$
is given by $\carries_{p}(\vec{c})$.
\label{theorem_Kummer}
\end{theorem}

\section{Divisibility by odd primes}
\label{section_divisibility}

First, we express the sum 
$A^{r}_{n}=\sum_{S\subseteq [n-1]} \beta_{n}(S)^{r}$ 
in terms of $\alpha_{n}(S)$.
\begin{lemma}
When $r$ is even, $A^{r}_{n}$ is given by
\begin{align}
A^{r}_{n}  	
& = 
\sum_{T_{1}, T_{2}, \ldots, T_{r} \subseteq [n-1]}
(-1)^{\sum_{i=1}^{r} |T_{i}|}
\cdot 
2^{n-1 - |\bigcup_{i=1}^{r} T_{i}|}
\cdot
\prod_{i=1}^{r}
\alpha_{n}(T_{i}) . 
\label{equation_expansion_r_even} \\
\intertext{When $r$ is odd, we have}
A^{r}_{n}  	
& = 
\sum_{\substack{T_{1}, T_{2}, \ldots, T_{r} \subseteq [n-1] \\
T_{1} \cup T_{2} \cupdots T_{r} = [n-1]}}
(-1)^{n-1 + \sum_{i=1}^{r} |T_{i}|}
\cdot 
\prod_{i=1}^{r}
\alpha_{n}(T_{i}) .
\label{equation_expansion_r_odd}
\end{align}
\label{lemma_expansion}
\end{lemma}
\begin{proof}
We begin by expanding $\beta_{n}(S)$ in terms of $\alpha_{n}(S)$:
\begin{align*}
A^{r}_{n}  	
& = 
\sum_{S \subseteq [n-1]} 
\beta_{n}(S)^{r} \\
& = 
\sum_{S \subseteq [n-1]}
\prod_{i=1}^{r}
\left(
\sum_{T_{i} \subseteq S}
(-1)^{|S-T_{i}|}
\cdot 
\alpha_{n}(T_{i})
\right)  \\
& = 
\sum_{T_{1}, T_{2}, \ldots, T_{r} \subseteq [n-1]}
\:
\sum_{T_{1} \cup T_{2} \cupdots T_{r} \subseteq S \subseteq [n-1]}
(-1)^{r \cdot |S|}
\cdot
(-1)^{\sum_{i=1}^{r} |T_{i}|}
\cdot 
\prod_{i=1}^{r}
\alpha_{n}(T_{i}) .
\end{align*}
When $r$ is even, we have $(-1)^{r \cdot |S|} = 1$,
and the inner sum has $2^{n-1 - |\bigcup_{i=1}^{r} T_{i}|}$ terms.
When $r$ is odd, the inner sum is zero unless
the union $\bigcup_{i=1}^{r} T_{i}$ is the whole set $[n-1]$.
\end{proof}

\begin{theorem}
Let $p$ be an odd prime and $r$ an even positive integer.
Assume that $m$ and $n$ contain the same non-zero digits
in base $p$.
Then the congruence 
$2^{-m} \cdot A^{r}_{m} \equiv 2^{-n} \cdot A^{r}_{n} \bmod p$ holds.
Especially,
the prime $p$ divides $A_{m}^{r}$ if and only if
$p$ divides $A_{n}^{r}$.
\label{theorem_mod_p_result}
\end{theorem}
\begin{proof}
Let $m$ and $n$ have the base $p$ expansions
$m = \sum_{j \geq 0} m_{j} \cdot p^{j}$
and
$n = \sum_{j \geq 0} n_{j} \cdot p^{j}$.
Then there exists a permutation $\pi$ on the 
non-negative integers such that
$m_{j} = n_{\pi(j)}$ for all $j \geq 0$.
Essentially, the permutation $\pi$ permutes
the powers of the prime $p$.
Define a bijection $f$ on the non-negative
integers by
$f\left(\sum_{j \geq 0} a_{j} \cdot p^{j}\right)
=
\sum_{j \geq 0} a_{j} \cdot p^{\pi(j)}$,
where $0 \leq a_{j} \leq p-1$.
Note that
$$
f(m)
=
\sum_{j \geq 0} m_{j} \cdot p^{\pi(j)}
=
\sum_{j \geq 0} n_{\pi(j)} \cdot p^{\pi(j)}
=
\sum_{j \geq 0} n_{j} \cdot p^{j}
=
n .
$$
Furthermore, when there are no carries
adding $x$ and $y$ in base $p$,
this function is additive, that is,
$f(x+y) = f(x) + f(y)$.
Also note that the inverse function $f^{-1}$
is additive under the same condition.
In terms of compositions, we have that
if $\vec{c} = (c_{1}, c_{2}, \ldots, c_{k})$
is a composition of~$m$
such that $\carries_{p}(\vec{c}) = 0$,
then the
composition
$f(\vec{c}) = (f(c_{1}), f(c_{2}), \ldots, f(c_{k}))$
is a composition of $f(m) = n$.

Let the non-carry power set $\NCP(m)$ be the collection
of all subsets of $[m-1]$ whose associated composition
has no carries when added in base $p$,
that is,
$$
\NCP(m)
=
\{ T \subseteq [m-1] \:\: : \:\: \carries_{p}(\co(T)) = 0 \} .
$$
Observe that $\NCP(m)$ is closed under inclusion.
Note that we can define
a bijection
$f : \NCP(m) \longrightarrow \NCP(n)$
by composing
the three maps
$$
\NCP(m)
\stackrel{\co}{\longrightarrow}
\{ \vec{c} \in \Comp(m) : \carries_{p}(\vec{c}) = 0 \}
\stackrel{f}{\longrightarrow}
\{ \vec{d} \in \Comp(n) : \carries_{p}(\vec{d}) = 0 \}
\stackrel{\co^{-1}}{\longrightarrow}
\NCP(n) .
$$
Since the compositions $\vec{c}$ and $f(\vec{c})$
have the same length,
the function $f$ preserves cardinality.
But there is a more direct description of the last map $f$
on sets. For
$T = \{t_{1} < t_{2} < \cdots < t_{k}\} \in \NCP(m)$,
we claim that
$f(T) = \{f(t_{1}), f(t_{2}), \ldots, f(t_{k})\}$.
Let $\vec{c}$ be the composition $\co(T)$.
By definition, the $i$th element of $f(T)$ is the initial partial sum
of the $i$ first elements of $f(\vec{c})$, that is,
$f(c_{1}) + \cdots + f(c_{i})$. Since the whole sum
$c_{1} + \cdots + c_{k}$ has no carries,
the partial sum also has no carries.
Hence, the $i$th element of $f(T)$ is given by
$f(c_{1}) + \cdots + f(c_{i})
=
f(c_{1} + \cdots + c_{i})
=
f(t_{i})$, proving the claim.

Also note that for a composition $\vec{c}$
without any carries, we have 
by Lucas' Theorem that
$$  \binom{m}{\vec{c}} \equiv \binom{f(m)}{f(\vec{c})} \bmod p , $$
since the factors of the product in Lucas' Theorem
are permuted by the permutation $\pi$.
Hence, for a set $T$ in $\NCP(m)$ we know that
$\alpha_{m}(T) = \alpha_{n}(f(T))$.

We now use the expansion in
equation~\eqref{equation_expansion_r_even}.
Let $\vec{c}\,^{i}$ be the composition
associated with the subset~$T_{i}$ of~$[m-1]$.
Similarly, let $U_{i}$ be
the subset of $[n-1]$ associated
with the composition
$f(\vec{c}\,^{i}) = \vec{d}\,^{i}$.
Next we study the two unions
$\bigcup_{i=1}^{r} T_{i}$
and
$\bigcup_{i=1}^{r} U_{i}$.
However, they may not be in the collection $\NCP(m)$, respectively, $\NCP(n)$.

For~$I$ a non-empty subset of the index set $[r]$ let
$T_{I}$ be the intersection
$\bigcap_{i \in I} T_{i}$.
Note that $T_{I}$ belongs to $\NCP(m)$
since this collection is closed under inclusion.
Similarly let
$U_{I}$ be the intersection
$\bigcap_{i \in I} U_{i}$
which belongs to $\NCP(n)$.
Note that $f(T_{I}) = U_{I}$
so the two sets $T_{I}$ and $U_{I}$ have the same cardinality.
By inclusion-exclusion we have
\begin{align*}
\left| \bigcup_{i=1}^{r} T_{i} \right|
& =
\sum_{\emptyset \subsetneqq I \subseteq [r]}
(-1)^{|I| - 1}
\cdot
|T_{I}|
=
\sum_{\emptyset \subsetneqq I \subseteq [r]}
(-1)^{|I| - 1}
\cdot
|U_{I}|
=
\left| \bigcup_{i=1}^{r} U_{i} \right| .
\end{align*}

Now observe that
the non-zero terms in
equation~\eqref{equation_expansion_r_even}
modulo $p$
are the terms where $T_{i}$ belongs to~$\NCP(m)$.
Hence, modulo~$p$ we have that
\begin{align*}
A^{r}_{m}  	
& \equiv 
\sum_{T_{1}, T_{2}, \ldots, T_{r} \in \NCP(m)}
(-1)^{\sum_{i=1}^{r} |T_{i}|}
\cdot 
2^{m-1 - |\bigcup_{i=1}^{r} T_{i}|}
\cdot
\prod_{i=1}^{r}
\alpha_{m}(T_{i}) \\
& \equiv 
2^{m-n}
\cdot
\sum_{U_{1}, U_{2}, \ldots, U_{r} \in \NCP(n)}
(-1)^{\sum_{i=1}^{r} |U_{i}|}
\cdot 
2^{n-1 - |\bigcup_{i=1}^{r} U_{i}|}
\cdot
\prod_{i=1}^{r}
\alpha_{n}(U_{i}) \\
& \equiv 
2^{m-n}
\cdot
A^{r}_{n}  	\bmod p .
\end{align*}
This proves the identity.
Finally, since $2$ is invertible modulo $p$, we obtain
that $A^{r}_{m}$ and $A^{r}_{n}$
either both have a factor of $p$ or
none of them have a factor of $p$.
\end{proof}

\begin{corollary}
When $r$ is even and $p$ is an odd prime,
the congruence
$A^{r}_{pn} \equiv A^{r}_{n} \bmod p$
holds.
\end{corollary}
\begin{proof}
Since $n$ and $p \cdot n$ have the same non-zero
digits modulo $p$, Theorem~\ref{theorem_mod_p_result} applies.
Hence, it is enough to observe that
$2^{pn} \equiv (2^{n})^{p} \equiv 2^{n} \bmod p$
using Fermat's little theorem.
\end{proof}

\begin{corollary}
When $r$ is even and $p$ is an odd prime, $A^{r}_{p^{k}}$ is not divisible by $p$.
\end{corollary}
\begin{proof}
It is enough to check that
$A^{r}_{1} = 1$ is not divisible by $p$.
\end{proof}

\begin{example}
{\rm
We can compute $A^{2}_{14}$ to observe that this number has
a factor of $3$. Hence by
Proposition~\ref{proposition_Euler_p}
we know that for all even $r$,
the prime $3$ divides $A^{r}_{14}$.
Furthermore, $14$ in base $3$ consists of
two $1$'s and one $2$.
Hence
Theorem~\ref{theorem_mod_p_result}
implies for
$n = 16, 22, 32, 34, 38, 42, 46, 48, 58, 64, 66, 86, 88, \ldots$
that $3$ divides $A^{r}_{n}$ as well.
}
\label{example_3_14}
\end{example}

\begin{example}
{\rm
Note that $5$ divides $A^{2}_{3} = 10$.
Hence, we know that $5$ divides $A^{4 \cdot i + 2}_{n}$ for $n$ 
of the form $3 \cdot 5^{k}$.
One may compute that
$5$ also divides $A^{2}_{12}$ and $A^{2}_{13}$.
This implies that $5$ divides $A^{4 \cdot i + 2}_{n}$ for $n$ belonging to
the following two sequences:
$12, 52, 60, 252, 260, 300, 1252, 1260, 1300, \ldots$
and
$13, 17, 53, 65, 77, 85, 253, 265, 325, 377, 385,  \ldots$.
}
\label{example_5}
\end{example}

\section{On the number of prime factors}
\label{section_number_of_prime_factors}

For a positive integer~$n$,
let $u_{p}(n)$ be the sum of the digits when $n$ is written in base $p$.
More formally,
for $n = \sum_{i \geq 0} n_{i} \cdot p^{i}$,
where $0 \leq n_{i} \leq p-1$,
the function $u_{p}(n)$ is given by the sum
$\sum_{i \geq 0} n_{i}$.
Furthermore, for a composition $\vec{c} = (c_{1}, c_{2}, \ldots, c_{k})$,
define $u_{p}(\vec{c})$ to be the sum of the digits
when all the parts of $\vec{c}$ are written in base $p$,
that is,
$u_{p}(\vec{c}) = \sum_{i=1}^{k} u_{p}(c_{i})$.
Finally, recall that
$\carries_{p}(\vec{c})$ denotes the number of carries when
adding $c_{1} + c_{2} + \cdots + c_{k}$ in base $p$.

\begin{lemma}
For a composition $\vec{c}$ of $n$,
the sum of its digits in base $p$ satisfies
$$ u_{p}(\vec{c}) = (p-1)\cdot \carries_{p}(\vec{c}) + u_{p}(n) . $$
\end{lemma}
\begin{proof}
If one lines up the parts $c_{1}, c_{2}, \dots, c_{k}$ of $\vec{c}$
in base $p$, note that any one of the $u_{p}(\vec{c})$ units
in any of these addends has only two options: 
It may either contribute to a carry 
along with another $p-1$ units in its column,
or it can directly become one of the $u_{p}(n)$ units in $n$.
\end{proof}

\begin{corollary}
Let $p$ be a prime. Then the number
of factors of $p$ in $A^{1}_{n}$ is 
$(n - u_{p}(n))/(p-1)$.
\end{corollary}
\begin{proof}
Note that $A^{1}_{n} = n! = \binom{n}{1,1, \ldots, 1}$.
Hence by Kummer's theorem the number of factors of
$p$ is
$\carries_{p}(1,1, \ldots, 1)
=
(u_{p}(1,1, \ldots, 1) - u_{p}(n))/(p-1)$.
\end{proof}

Similarly, define the depth of $n$ 
to be $d_{p}(n) = u_{p}(n) - 1$,
that is, the sum of the digits of $n$ in base~$p$
beyond the requisite digit greater than zero in its first position.
Further, define the depth of a composition
$\vec{c}$ 
to be the sum of the depth of
each of its parts,
that is,
$d_{p}(\vec{c}) = \sum_{i=1}^{k} d_{p}(c_{i})$.
The next lemma is direct.
\begin{lemma}
For a composition $\vec{c}$ into $k$ parts,
$u_{p}(\vec{c}) = d_{p}(\vec{c}) + k$.
\end{lemma}

Recall according to the map $\co$ from subsets $S \subseteq [n-1]$
to compositions $\vec{c}$ of $n$ that the number of parts $k$
of $\vec{c}$ is one more than the cardinality of $S$. 
Combining this observation with the previous two lemmas yields the next result.

\begin{proposition}
For a set $S \subseteq [n-1]$ and its associated composition 
$\co(S) = \vec{c}$ of $n$,
the number of carries
$\carries_{p}(\vec{c})$ is given by
$(d_{p}(\vec{c}) + |S| - d_{p}(n))/(p-1)$.
\label{proposition_Carries}
\end{proposition}

This gives way to the main result in this section.

\begin{theorem}
When $r$ is odd and $p$ is prime, the sum $A^{r}_{n}$ 
has at least 
$(n - 1 - r \cdot d_{p}(n))/(p-1)$
factors of $p$.
\label{theorem_odd_prime}
\end{theorem}
\begin{proof}
Consider a term in equation~\eqref{equation_expansion_r_odd},
where we let $\co(T_{i}) = \vec{c}\,^{i}$.
The number of factors of $p$ in this term is given by
\begin{align*}
\sum_{i=1}^{r} \carries_{p}(\vec{c}\,^{i})
& =
\sum_{i=1}^{r} \frac{d_{p}(\vec{c}\,^{i}) + |T_{i}| - d_{p}(n)}{p-1} 
\geq
\sum_{i=1}^{r} \frac{|T_{i}| - d_{p}(n)}{p-1} 
\geq
\frac{n-1 - r\cdot d_{p}(n)}{p-1},
\end{align*}
since $d_{p}(\vec{c}\,^{i}) \geq 0$ for all $i$
and
$\sum_{i=1}^{r} |T_{i}| \geq |\bigcup_{i=1}^{r} T_{i}| = n-1$.
\end{proof}

We can say something stronger when the prime $p$ is $2$.

\begin{theorem}
The sum $A^{r}_{n}$ is divisible by $2^{n - 1 - r \cdot d_{2}(n)}$.
\label{theorem_even_prime}
\end{theorem}
\begin{proof}
The case when $r$ is odd follows from Theorem~\ref{theorem_odd_prime}.
Now suppose $r$ is even, and consider a term in
equation~\eqref{equation_expansion_r_even}.
The number of factors of $2$ in this term is given by
\begin{align*}
n-1 - \left|\bigcup_{i=1}^{r} T_{i}\right| + \sum_{i=1}^{r} \carries_{2}(\vec{c}\,^{i})
& \geq
n-1 - \sum_{i=1}^{r} |T_{i}| + \sum_{i=1}^{r} \carries_{2}(\vec{c}\,^{i}) \\
& =
n - 1 - \sum_{i=1}^{r} |T_{i}| + 
\sum_{i=1}^{r} (d_{2}(\vec{c}\,^{i}) + |T_{i}| - d_{2}(n)) \\
& \geq
n-1 - r\cdot d_{2}(n).
\qedhere
\end{align*}
\end{proof}

Since $d_{2}(2^{k}) = 0$, we have the following corollary.
\begin{corollary}
When $n$ is a power of $2$,
then $A^{r}_{n}$ is divisible by
$2^{n-1}$.
\label{corollary_two_power}
\end{corollary}

In this case, we actually have equality.
\begin{proposition}
When $n$ is a power of $2$, then $2^{n-1}$
is the highest power dividing $A_{n}^{r}$.
\end{proposition}
\begin{proof}
For $x$
an $r$-tuple 
$(T_{1}, T_{2}, \ldots, T_{r})$,
let $h(x)$ denote the associated term
in 
Lemma~\ref{lemma_expansion}.
Note that the expression for $h(x)$ depends on the
parity of $r$.
Observe in the proofs of
Theorems~\ref{theorem_odd_prime}
and~\ref{theorem_even_prime}
that we have equality in the bound for those terms
where the sets $T_{i}$ are disjoint and $d_{2}(\vec{c}\,^{i})=0$
for all $i$. Note that the latter condition requires
all the parts of $\vec{c}\,^{i}$ to be powers of $2$.
Let $X$ be the collection of all such $r$-tuples.

Consider an $r$-tuple
$x=(T_{1}, T_{2}, \ldots, T_{r}) \in X$,
and choose the smallest 
$1\leq k\leq \left\lfloor{r}/{2}\right\rfloor$ such that 
$T_{2k-1} \neq T_{2k}$, if one exists.
Let $x'$ be obtained by switching the
$(2k-1)$st and $2k$th subsets,
that is,
$x'=(T_{1},\ldots,T_{2k-2}, T_{2k}, T_{2k-1},T_{2k+1},\ldots,T_{r})$.
Observe that $h(x') = h(x)$.
Since the subsets~$T_{i}$ are all disjoint, the only
case where such a $k$ does not exist is when 
$T_{1}, T_{2}, \ldots, T_{2\cdot\lfloor{r}/{2}\rfloor}$ are all empty. 
This occurs in a single $r$-tuple $x_{0}$, where 
$x_{0} = (\emptyset, \emptyset, \ldots, \emptyset)$ if $r$ is even, and
$x_{0} = (\emptyset, \emptyset, \ldots, \emptyset, [n-1])$ if $r$ is odd. 
When $r$ is even, we directly observe that
$h(x_{0}) \equiv 2^{n-1} \bmod 2^{n}$.
When $r$ is odd we have
$h(x_{0}) \equiv n! \equiv 2^{n-1} \bmod 2^{n}$,
using that $n$ is a power of $2$.
Now, after pairing up all these terms except the term $h(x_{0})$, 
the result follows by
\begin{align*}
A_{n}^{r} 
& \equiv
\sum_{x \in X} h(x)
\equiv
h(x_{0})
\equiv
2^{n-1}
\bmod 2^{n} .
\qedhere
\end{align*}
\end{proof}

\section{Improving the bound}
\label{section_improving_the_bound}

We now improve upon the bounds of
Theorems~\ref{theorem_odd_prime}
and~\ref{theorem_even_prime}.

\begin{proposition}
When $p$ is an odd prime and $n\geq 2$, 
the sum $A_{n}^{p^k}$ has at least
\[
\frac{n-1-p^k \cdot d_p(n)}{p-1} + k
\]
factors of $p$.
\label{proposition_simple_shift}
\end{proposition}
\begin{proof}
Consider the term indexed by the $p^k$-tuple
$(T_1, T_2, \ldots, T_{p^k})$
in equation~\eqref{equation_expansion_r_odd},
and consider the action by the shift
$(T_2, T_3, \ldots, T_{p^k}, T_1)$. 
Note that the size of the orbit of this action
is $p^i$, for some  $0\leq i \leq k$.
Grouping these $p^{i}$ identical terms
gives $i$ factors of $p$. 
However, this means that our tuple
$(T_1, T_2, \ldots, T_{p^k})$
can be written as
$(T_1, T_2, \ldots, T_{p^i},
T_1, T_2, \ldots, T_{p^i},\ldots,
T_1, T_2, \ldots, T_{p^i})$
up to a cyclic shift, and further,
$\bigcup_{j=1}^{p^i} T_j = [n - 1]$.
Hence, the number of factors of $p$
in these terms is
\begin{align*}
i + \sum_{j=1}^{p^k} \carries_p(\vec{c}\,^j)
& = 
i + \sum_{j=1}^{p^k} \frac{d_p(\vec{c}\,^j) + |T_j| - d_p(n)}{p-1} \\
& \geq 
i
+
\frac{\left(\sum_{j=1}^{p^k} |T_j|\right) - p^k\cdot d_p(n)}{p-1} \\
& \geq
i + \frac{p^{k-i} \cdot (n-1) - p^k\cdot d_p(n)}{p-1} \\
& = 
i
+
\frac{(p^{k-i} - 1) \cdot (n-1)}{p-1}
+
\frac{n - 1 - p^k\cdot d_p(n)}{p-1} \\
& \geq
k + \frac{n - 1 - p^k\cdot d_p(n)}{p-1},
\end{align*}
where in the last step we used
$(p^{k-i} - 1)/(p-1) = 1 + p + \cdots + p^{k-i-1} \geq k-i$
and
$n-1 \geq 1$.
\end{proof}

The above proof uses the action of the cyclic group
${\mathbb Z}_{p^{k}}$
to collect terms together.
We can improve the bound of Proposition~\ref{proposition_simple_shift}
in some cases by using a larger group acting on the $r$-tuples.

Let $q$ be the prime power $p^{k}$. 
We define the group $G_{q}$ acting on the set $[q]$.
The generators are indexed by pairs $(a, b)$ where 
$1 \leq a \leq k$
and
$0\leq b \leq p^{k - a} - 1$.
The generator $\sigma_{a, b}$ is given 
by the following product of $p$-cycles,
\[
\sigma_{a, b} = 
\prod_{i=1}^{p^{a-1}} 
(
i + bp^{a}, 
i + bp^{a} + p^{a-1}, 
\dots, 
i + bp^{a} + (p-1) p^{a-1}
).
\] 

To give a geometric picture of the action of this group,
consider a balanced $p$-ary tree of depth~$k$. 
This tree has $q$ leaves, which we label $1$ through $q$.
Furthermore, the tree has $\frac{q-1}{p-1}$ internal nodes,
which are indexed by the pairs $(a, b)$. 
The $a$ coordinate states that the internal node is
at depth $k-a$. The $b$ coordinate indicates which
node at that depth, reading from left to right.
The generator $\sigma_{a, b}$ then cyclically shifts
the $p$ children of this node. See Figure~\ref{Tree}
for an example.
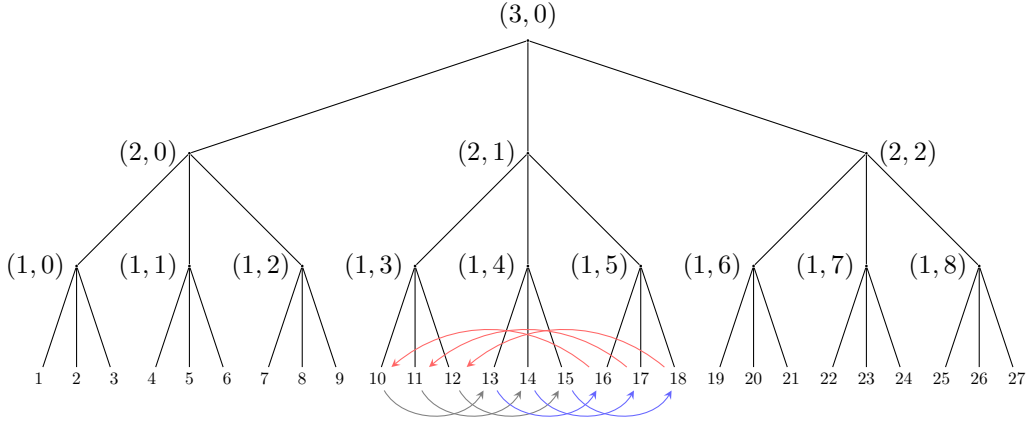
\begin{figure}
	\[\begin{tikzpicture}[scale=1, >=stealth] 
		\tikzstyle{internal node}=[circle, inner sep=0.3,fill=black]

		\tikzstyle{level 1}=[sibling distance=45mm]
		\tikzstyle{level 2}=[sibling distance=15mm]
		\tikzstyle{level 3}=[sibling distance=5mm]
		\node[internal node, label=above:{\small$(3, 0)$}](0){}
			child{node[internal node, label=left:{\small$(2, 0)$}]{}
				child{node[internal node, label=left:{\small$(1, 0)$}]{}
					child{node[scale=0.6]{${1}$}}
					child{node[scale=0.6]{${2}$}}
					child{node[scale=0.6]{${3}$}}
				}
				child{node[internal node, label=left:{\small$(1,1)$}]{}
					child{node[scale=0.6]{${4}$}}
					child{node[scale=0.6]{${5}$}}
					child{node[scale=0.6]{${6}$}}
				}
				child{node[internal node, label=left:{\small$(1, 2)$}]{}
					child{node[scale=0.6]{${7}$}}
					child{node[scale=0.6]{${8}$}}
					child{node[scale=0.6]{${9}$}}
				}
			}
			child{node[internal node, label=left:{\small$(2, 1)$}]{}
				child{node[internal node, label=left:{\small$(1, 3)$}]{}
					child{node[scale=0.6]{${10}$}}
					child{node[scale=0.6]{${11}$}}
					child{node[scale=0.6]{${12}$}}
				}
				child{node[internal node, label=left:{\small$(1, 4)$}]{}
					child{node[scale=0.6]{${13}$}}
					child{node[scale=0.6]{${14}$}}
					child{node[scale=0.6]{${15}$}}
				}
				child{node[internal node, label=left:{\small$(1, 5)$}]{}
					child{node[scale=0.6]{${16}$}}
					child{node[scale=0.6]{${17}$}}
					child{node[scale=0.6]{${18}$}}
				}
			}
			child{node[internal node, label=right:{\small$(2, 2)$}]{}
				child{node[internal node, label=left:{\small$(1, 6)$}]{}
					child{node[scale=0.6]{${19}$}}
					child{node[scale=0.6]{${20}$}}
					child{node[scale=0.6]{${21}$}}
				}
				child{node[internal node, label=left:{\small$(1, 7)$}]{}
					child{node[scale=0.6]{${22}$}}
					child{node[scale=0.6]{${23}$}}
					child{node[scale=0.6]{${24}$}}
				}
				child{node[internal node, label=left:{\small$(1, 8)$}]{}
					child{node[scale=0.6]{${25}$}}
					child{node[scale=0.6]{${26}$}}
					child{node[scale=0.6]{${27}$}}
				}
			};
			
		\draw[gray, ->, bend right=60pt] (0-2-1-1) to (0-2-2-1);
		\draw[gray, ->, bend right=60pt] (0-2-1-2) to (0-2-2-2);
		\draw[gray, ->, bend right=60pt] (0-2-1-3) to (0-2-2-3);
		
		\draw[blue!60!white, ->, bend right=60pt] (0-2-2-1) to (0-2-3-1);
		\draw[blue!60!white, ->, bend right=60pt] (0-2-2-2) to (0-2-3-2);
		\draw[blue!60!white, ->, bend right=60pt] (0-2-2-3) to (0-2-3-3);
		
		\draw[red!60!white, ->, bend right=40pt] (0-2-3-1) to (0-2-1-1);
		\draw[red!60!white, ->, bend right=40pt] (0-2-3-2) to (0-2-1-2);
		\draw[red!60!white, ->, bend right=40pt] (0-2-3-3) to (0-2-1-3);
	  \end{tikzpicture}\]
\caption{A balanced ternary tree of depth $3$ with
the action of $\sigma_{2, 1}$ shown.}
\label{Tree}
\end{figure}

With this geometric picture, it is straightforward
to observe that the group has order 
$p^{\frac{q-1}{p-1}}$. 
Given a $q$-tuple of sets 
$x=(T_{1}, T_{2}, \dots, T_{q})$,
let the group $G_{q}$ act on $x$ by
permuting the indices.
Let $\Orb_{x}$ be the orbit of the $q$-tuple $x$,
that is, $\Orb_{x} = \{g \cdot x : g \in G_{q}\}$.
Note that the cardinality of
the orbit $\Orb_{x}$ is a power of $p$.

Additionally, for an $r$-tuple
$x = (T_{1}, \ldots, T_{r})$
let
$f_{n}^{r}(x)=
(-1)^{\sum_{i=1}^{r} |T_{i}|}
\cdot
\prod_{i=1}^{r} \alpha_{n}(T_{i})$.

\begin{proposition}
Let $q=p^{k}$ and $d_{p}(n) > 0$. For a $q$-tuple $x=(T_{1}, T_{2}, \dots, T_{q})$, 
the sum
$
\displaystyle
\sum_{y\in \Orb_x} f_{n}^{q}(y)
$
has at least
\[
\frac{q-1 + \left|\bigcup_{i=1}^{q} T_{i} \right| - q\cdot d_{p}(n)}{p-1}
\]
factors of $p$.
\label{proposition_q}
\end{proposition}
\begin{proof}
The proof is by induction on $k$.
The induction basis is $k=0$, that is, $q=1$.
Here $\Orb_{x}$ consists only of $(T)$.
The number of $p$-factors are
\begin{align*}
\carries_{p}(\vec{c}) 
&= 
\frac{d_{p}(\vec{c}) + |T| - d_{p}(n)}{p-1}
\geq
\frac{\left| T \right| - d_{p}(n)}{p-1},
\end{align*}
since $d_{p}(\vec{c})\geq 0$,
which completes the basis of the induction.

Now assume that the statement is true for all
$p$-powers strictly less than $q$.
Notice that $f_{n}(y) = f_{n}(x)$ for all $y \in \Orb_{x}$. Hence,
\begin{align*}
\sum_{y \in \Orb_{x}} f_{n}(y) 
&= 
|\Orb_{x}| \cdot f_{n}^{q}(x)
=
|\Orb_{x}| \cdot (-1)^{\sum_{i=1}^{q} |T_{i}|}
\cdot \prod_{i=1}^{q} \alpha_{n}(T_{i}).
\end{align*}
Furthermore, the number of factors of $p$ in the last expression is
\begin{align*}
\log_{p}(|\Orb_{x}|)
+
\sum_{i=1}^{q} \carries_{p}(\vec{c}\,^{i}) 
&= 
\log_{p}(|\Orb_{x}|)
+
\sum_{i=1}^{q} \frac{d_{p}(\vec{c}\,^{i}) + |T_{i}| - d_{p}(n)}{p-1} .
\end{align*}

For $0 \leq b \leq p-1$ let
$x_{b}$ denote the $q/p$-tuple
$(T_{b \cdot q/p + 1},\ldots,T_{(b+1) \cdot q/p})$,
that is, the $q/p$-tuple of sets below the
node $(k-1,b)$ in the tree.

First, assume that the stabilizer of $x$
contains an element involving the
permutation~$\sigma_{k,0}$.
That is, the stabilizer contains a rotation 
centered at the root~$(k,0)$ of the tree.
Then the leaves below
the nodes $(k-1,0)$ are the same as the leaves
below $(k-1,b)$.
Then the cardinality of the orbit
$\Orb_{x}$ is the same
as the size of the orbit
$\Orb_{x_{0}}$.
Hence we can apply the induction hypotheses
to the node $(k-1,0)$ of the tree:
\begin{align*}
\log_{p}(|\Orb_{x}|)
+
\sum_{i=1}^{q} \carries_{p}(\vec{c}\,^{i})
&= 
\log_{p}(|\Orb_{x_{0}}|)
+
\sum_{i=1}^{q/p} \carries_{p}(\vec{c}\,^{i}) 
+
\sum_{i=q/p+1}^{q} \carries_{p}(\vec{c}\,^{i})  \\
& \geq
\frac{q/p-1 + \left|\bigcup_{i=1}^{q/p} T_{i} \right| - q/p \cdot d_{p}(n)}{p-1}
+
\sum_{i=q/p+1}^{q} \frac{d_{p}(\vec{c}\,^{i}) + |T_{i}| - d_{p}(n)}{p-1}  \\
& =
\frac{q/p-1 + \left|\bigcup_{i=1}^{q} T_{i} \right| - q \cdot d_{p}(n)}{p-1}
+
\sum_{i=q/p+1}^{q} \frac{d_{p}(\vec{c}\,^{i}) + |T_{i}|}{p-1} .
\end{align*}
If $T_{i}$ is non-empty, then
$|T_{i}| \geq 1$.
If $T_{i}$ is empty, then
$\vec{c}\,^{i}$ is the composition $n$,
so
$d_{p}(\vec{c}\,^{i}) = d_{p}(n) \geq 1$
by our assumption.
In both cases we have 
$d_{p}(\vec{c}\,^{i}) + |T_{i}| \geq 1$
for all $q/p+1 \leq i \leq q$.
Thus, we can apply this inequality
\begin{align*}
\log_{p}(|\Orb_{x}|)
+
\sum_{i=1}^{q} \carries_{p}(\vec{c}\,^{i})
& \geq
\frac{q/p-1 + \left|\bigcup_{i=1}^{q} T_{i} \right| - q \cdot d_{p}(n)}{p-1}
+
\frac{q - q/p}{p-1} ,
\end{align*}
which yields the bound.

It remains to consider the case when
the stabilizer of $x$ does not contain
a rotation centered at the root~$(k,0)$.
Now the cardinality of the orbit of $x$ is given
by the product
$$
|\Orb_{x}|
=
\prod_{b=0}^{p-1}
|\Orb_{x_{b}}| .
$$
Hence we apply the induction hypotheses 
to each child of the root
\begin{align*}
\log_{p}(|\Orb_{x}|)
+
\sum_{i=1}^{q} \carries_{p}(\vec{c}\,^{i})
& =
\sum_{b=0}^{p-1}
\left(
\log_{p}(|\Orb_{x_{b}}|)
+
\sum_{i=1}^{q/p} \carries_{p}(\vec{c}\,^{b \cdot q/p + i})
\right) \\
& \geq
\sum_{b=0}^{p-1}
\frac{q/p-1 + \left|\bigcup_{i=1}^{q/p} T_{b \cdot q/p + i} \right| - q/p\cdot d_{p}(n)}
{p-1} \\
& \geq
\frac{q-p + \left|\bigcup_{i=1}^{q} T_{i} \right| - q\cdot d_{p}(n)}{p-1} ,
\end{align*}
which yields the bound.
This completes the second case and the induction.
\end{proof}

\begin{theorem}
For $r$ odd and $d_{p}(n)>0$, 
the sum $A_{n}^{r}$ contains
at least
$$
\left\lceil
\frac{r - u_{p}(r) + n - 1 - r\cdot d_{p}(n)}{p - 1}
\right\rceil
$$
factors of $p$.
\label{theorem_group_odd_p}
\end{theorem}
\begin{proof}
Let $r = \sum_{i=1}^{u_{p}(r)} q_{i}$ where $q_{i}$ is a power of $p$.
Note that a power $p^{j}$ occurs at most $p-1$ times in
this sum.
Now define the group
$G$ to be the Cartesian product
$G = \prod_{i=1}^{u_{p}(r)} G_{q_{i}}$.
Furthermore, let $G$ act on the set $[r]$
by letting the $G_{q_{i}}$ act on the interval
$[q_{1} + \cdots + q_{i-1} + 1, q_{1} + \cdots + q_{i-1} + q_{i}]$.
The action of the group $G$ can be viewed as forest consisting of
$u_{p}(r)$ trees.
Finally, let $G$ act on a $r$-tuple by acting on the indices of
the tuple.

Note that the function $f_{n}^{r}$ is multiplicative in the following
meaning.
For an $r$-tuple $x = (T_{1}, \ldots, T_{r})$
define
$x_{i}$ to be the 
$q_{i}$-tuple $(T_{q_{1} + \cdots + q_{i-1} + 1},
\ldots,
T_{q_{1} + \cdots + q_{i-1} + q_{i}})$.
Then we have
$$
f_{n}^{r}(T_{1}, \ldots, T_{r})
=
\prod_{i=1}^{u_{p}(r)} 
f_{n}^{q_{i}}(x_{i}) . $$
Now the sum over an orbit of the $r$-tuple
$x = (T_{1}, \ldots, T_{r})$
factors as
$$
\sum_{y\in \Orb_x} f_{n}^{r}(y)
=
\prod_{i=1}^{u_{p}(r)} 
\sum_{y_{i} \in \Orb_{x_{i}}} f_{n}^{q_{i}}(y_{i}) .
$$
Hence we can apply
Proposition~\ref{proposition_q}
to each factor,
and the sum over the orbit has at least
\begin{align}
&
\sum_{i=1}^{u_{p}(r)}
\frac{1}{p-1}
\cdot
\left(
q_{i} - 1
+
\left|
\bigcup_{j=q_{1}+\cdots+q_{i-1}+1}^{q_{1}+\cdots+q_{i}} T_{j}
\right|
-
q_{i} \cdot d_{p}(n)
\right)
\nonumber \\
&
\geq
\frac{1}{p - 1}
\cdot
\left(
r - u_{p}(r)
+
\left|
\bigcup_{j=1}^{r} T_{j}
\right|
-
r \cdot d_{p}(n)
\right) 
\label{equation_needed_in_the_next_proof}
\\
&
=
\frac{1}{p - 1}
\cdot
\left(
r - u_{p}(r)
+
n-1
-
r \cdot d_{p}(n)
\right) ,
\nonumber
\end{align}
where the last equality comes from
the assumption
$T_{1} \cup T_{2} \cupdots T_{r} = [n-1]$
in equation~\eqref{equation_expansion_r_odd}.
\end{proof}

Again, we can make a stronger statement when $p=2$.
\begin{theorem}
For $d_{2}(n)>0$, 
the sum $A_{n}^{r}$ contains at least
$r - u_{2}(r) + n - 1 - r\cdot d_{2}(n)$
factors of $2$.
\label{theorem_better_bound_2}
\end{theorem}
\begin{proof}
The case where $r$ is odd follows from 
Theorem~\ref{theorem_group_odd_p}.
We retain the notation of the proof of
Theorem~\ref{theorem_group_odd_p}.
Note that in that proof, we did not use
the parity of $r$ until the very end.
Now assume that $r$ is even.
For an $r$-tuple $x = (T_{1}, \ldots, T_{r})$
define the function
$g_{n}^{r}(x)
=
2^{n-1 - |\bigcup_{i=1}^{r} T_{i}|}
\cdot f_{n}^{r}(x)$,
which is the expression
in equation~\eqref{equation_expansion_r_even}.
Hence the number of factors of
$2$ in the sum over the orbit
$$
\sum_{y\in \Orb_x} g_{n}^{r}(y)
=
2^{n-1 - |\bigcup_{i=1}^{r} T_{i}|}
\cdot
\sum_{y\in \Orb_x} f_{n}^{r}(y)
$$
is bounded from below by
the sum
of
$n-1 - \left|\bigcup_{i=1}^{r} T_{i}\right|$
and the expression~\eqref{equation_needed_in_the_next_proof}.
That is,
\begin{align*}
n-1 - \left|\bigcup_{i=1}^{r} T_{i}\right|
+
r - u_{2}(r)
+
\left|
\bigcup_{j=1}^{r} T_{j}
\right|
-
r \cdot d_{2}(n)
& =
n-1
+
r - u_{2}(r)
-
r \cdot d_{2}(n) .
\qedhere
\end{align*}
\end{proof}

Theorem~\ref{theorem_better_bound_2}
improves upon Theorem~\ref{theorem_even_prime}
by at least $1$ when $n$ is not a $2$-power.

\begin{corollary}
When $n$ is not a power of $2$
and $r \geq 2$, then $A^{r}_{n}$
has at least $n - r \cdot d_{2}(n)$ factors of~$2$.
\label{corollary_not_a_two_power}
\end{corollary}
\begin{proof}
Note that $r \geq 2$ implies that $r > u_{2}(r)$,
that is,
$r - u_{2}(r) - 1 \geq 0$.
Hence by Theorem~\ref{theorem_better_bound_2}
we have
$r - u_{2}(r) + n - 1 - r \cdot d_{2}(n)
\geq
n - r \cdot d_{2}(n)$.
\end{proof}

\TablePisTwo

\TablePisThree

\begin{corollary}
Let $n$ satisfy the inequality $2^{k} \leq n \leq 2^{k + 1} - 1$.
Then $A_n^2$ is divisible by $2^{2^{k} - 1}$.
\end{corollary}
\begin{proof}
Write $n$ as the sum $2^{k} + a$.
When $a=0$ there is nothing to prove
by Corollary~\ref{corollary_two_power}.
When $a \geq 1$ we have
$d_{2}(n) = u_{2}(a)$.
Furthermore, since for each $2$-power
$2^{j}$, where $j \geq 1$, we have
$2^{j} - 2 \cdot u_{2}(2^{j}) \geq 0$.
But for $j=0$ we have
$2^{j} - 2 \cdot u_{2}(2^{j}) = -1$.
Hence for all non-negative $a$ we have
$a - 2 \cdot u_{2}(a) \geq -1$.
Hence the bound by
Corollary~\ref{corollary_not_a_two_power}
yields
$n - 2 \cdot d_{2}(n)
=
2^{k} + a - 2 \cdot u_{2}(a) \geq 2^{k} - 1$.
\end{proof}

\section{Concluding remarks}

Some of the results in this paper are reminiscent of results in the
papers~\cite{C_E_P_R,Ehrenborg_Fox_1,Ehrenborg_Fox_2},
where there are results which depend on the binary expansion
of the parameters.
However, as the reader can see from
Tables~\ref{table_p_2} and~\ref{table_p_3},
where we present computational results
for the numbers of factors of the primes $2$ and~$3$
in $A^{r}_{n}$, a lot of work remains
in order to understand these numbers.

A final question is to understand the asymptotic behavior
of $A^{r}_{n}$ as $n$ tends to infinity. How similar is this
behavior to Stirling's formula?

\section*{Acknowledgments}

The authors thank the referee for comments on an earlier version of this paper.
This work was supported by a grant from the Simons Foundation
(\#429370, Richard~Ehrenborg).

\newcommand{\journal}[6]{{\sc #1,} #2, {\it #3} {\bf #4} (#5) #6.}
\newcommand{\book}[4]{{\sc #1,} #2, #3, #4.}

\bigskip

{\em R.\ Ehrenborg, A.\ Happ.
Department of Mathematics,
University of Kentucky,
Lexington, KY 40506-0027,}
{\tt richard.ehrenborg@uky.edu},
{\tt alex.happ@uky.edu}

\end{document}